\documentclass[11pt,a4paper]{amsart}
\usepackage[T1]{fontenc}
\usepackage{amsmath,amssymb}
\newtheorem{theorem}{Theorem}
\usepackage[margin=1.1in]{geometry}
\usepackage{booktabs}
\usepackage{url}
\newtheorem{lemma}[theorem]{Lemma}
\newtheorem{proposition}[theorem]{Proposition}
\newtheorem{corollary}[theorem]{Corollary}
\newcommand{\Q}{\mathbb Q}
\newcommand{\Z}{\mathbb Z}
\newcommand{\rk}{\operatorname{rk}}
\newcommand{\Sha}{\operatorname{\mathrm{Sha}}}
\title{Elimination of Ren\'e Peschmann's 968 Remaining Hard Fibers for the Perfect Cuboid Problem}
\author{Ricky Cipollini}

\begin{document}
\begin{abstract}
Recently, Ren\'e Peschmann proved perfect-cuboid nonexistence on 1,072 explicit master-tuple fibers with $\max(m,n)\le100$, leaving 968 fibers in the same bounded parameter range. In this paper, I eliminate all 968 remaining fibers. Most fibers fall to fairly short rank-zero or covering arguments, while the final cases need the stronger local-global sieves. Together with Peschmann's theorem this excludes perfect cuboids on all 2,040 admissible fibers with $m\le100$; the result is still a bounded-fiber theorem rather than a proof of the full perfect-cuboid conjecture.
\end{abstract}

\maketitle
\markboth{RICKY CIPOLLINI}{ELIMINATION OF PESCHMANN'S 968 REMAINING FIBERS}
\section{Setup and reduction}

This is a direct follow-up to Peschmann's 1,072-fiber paper \cite{Pes1072}. I keep his master-tuple language throughout, since it makes the relation between the two papers very clear. His earlier paper \cite{PesQuartic} gives the genus-three reduction and elliptic quotients, and \cite{Pes1072} uses them to eliminate 1,072 explicit fibers. Here I take the complementary 968 fibers in exactly the same bounded ledger and eliminate all of them. The complete computational package for this paper can be downloaded at https://github.com/mrricky22/968-fibers

To keep the paper reasonably short I do not print the full machine output here. The package contains the exact inputs, scripts, outputs and fiber-by-fiber data used below, while the mathematical reductions and the meaning of each computation are given in the paper.

An \emph{Euler brick} is a triple $(X,Y,Z)\in\Z_{>0}^3$ for which $X^2+Y^2$, $X^2+Z^2$ and $Y^2+Z^2$ are squares. It is a \emph{perfect cuboid} if $X^2+Y^2+Z^2$ is also a square. After dividing by $\gcd(X,Y,Z)$ we may work primitively. Modulo $4$, two odd edges are impossible, while three even edges contradict primitivity, so a primitive Euler brick has exactly one odd edge. Call it $X$.

For coprime opposite-parity pairs $a>b>0$ and $m>n>0$, put
\[
U_1=a^2-b^2,\quad V_1=2ab,\quad W_1=a^2+b^2,\qquad
U=m^2-n^2,\quad V=2mn,\quad W=m^2+n^2.
\]
Thus $U_1^2+V_1^2=W_1^2$ and $U^2+V^2=W^2$. Following Peschmann, call $(a,b,m,n)$ a master tuple when
\[
M=(V_1U)^2+(U_1V)^2
\]
is a square. With $g=\gcd(U_1,U)$ the corresponding primitive brick has
\[
(X,Y,Z)=\frac1g(U_1U,V_1U,U_1V),
\]
and its space diagonal is rational exactly when
\[
F=(W_1U)^2+(U_1V)^2
\]
is a square. The scaling by $g$ clearly does not change rational squareness.

\begin{proposition}[master-tuple reduction]
Every primitive Euler brick, after naming its unique odd edge $X$ and ordering the other two edges, comes from a unique master tuple as above. Swapping the two even edges swaps the two Euclid parameter pairs.
\end{proposition}

\begin{proof}
Let $d=\gcd(X,Y)$ and $e=\gcd(X,Z)$. The triples obtained from the two faces containing the odd edge are primitive Pythagorean triples, so uniquely
\[
X/d=U_1,\quad Y/d=V_1,\qquad X/e=U,\quad Z/e=V
\]
for coprime opposite-parity Euclid pairs. Any common divisor of $d$ and $e$ divides $X,Y,Z$, hence $\gcd(d,e)=1$. From $dU_1=eU$ we therefore get $d=U/g$ and $e=U_1/g$ with $g=\gcd(U_1,U)$. Hence the displayed brick formula follows. Finally $V_1U=gY$ and $U_1V=gZ$, so $M=g^2(Y^2+Z^2)$ is a square because the third face diagonal is integral. Uniqueness is just uniqueness of the primitive Pythagorean parametrisations.
\end{proof}

Let
\[
\mathcal P_{100}=\{(m,n):2\le m\le100,\ 1\le n<m,\ \gcd(m,n)=1,\ m-n\ {\rm odd}\}.
\]
There are $|\mathcal P_{100}|=2040$ such pairs. Peschmann's explicit set $\mathcal S_{\rm Pes}$ contains 1,072 of them, and I write
\[
\mathcal R_{968}:=\mathcal P_{100}\setminus\mathcal S_{\rm Pes}
\]
for the complementary set of 968 pairs. Thus $\mathcal P_{100}=\mathcal S_{\rm Pes}\sqcup\mathcal R_{968}$. A fixed $(m,n)$ still allows unbounded $(a,b)$, which is worth keeping in mind: eliminating one fiber is an infinite statement in the other parameter pair, not an edge-height search.

\begin{theorem}\label{thm:main}
For every $(m,n)\in\mathcal R_{968}$ there is no positive master-tuple specialisation $(a,b,m,n)$ giving a perfect cuboid.
\end{theorem}

\begin{corollary}\label{cor:2040}
Taking Peschmann's 1,072-fiber theorem \cite{Pes1072} together with Theorem~\ref{thm:main}, no perfect cuboid occurs on any fiber $(m,n)\in\mathcal P_{100}$.
\end{corollary}

One point is worth stressing. Corollary~\ref{cor:2040} is not the global perfect-cuboid conjecture: nothing here forces an arbitrary master tuple to have $m\le100$. The result covers all 2,040 fibers in this explicit bounded parameter range, and no more is claimed.

It is also useful to put Peschmann's genus-three curve next to the one used later. Write $r=a/b$ and
\[
P(s)=V^2s^2+(4U^2-2V^2)s+V^2,\qquad
Q(s)=W^2s^2+2(U^2-V^2)s+W^2.
\]
His fiber is
\[
H_{m,n}:\quad v^2=P(r^2)Q(r^2).
\]
The exact discriminants and resultant are
\[
\operatorname{disc}(P)=16U^2(U^2-V^2),\qquad
\operatorname{disc}(Q)=-16U^2V^2,\qquad
\operatorname{Res}_s(P,Q)=16U^8.
\]
The last expression corrects the printed resultant $256U^4V^4$ in \cite{Pes1072}. This typo does not affect nonvanishing, smoothness, or the genus-three conclusion. The involutions $r\mapsto-r$, $r\mapsto1/r$ and $r\mapsto-1/r$ give the elliptic quotients
\[
E_{uV}:\ Y^2=(V^2u^2+4(U^2-V^2))(W^2u^2-4V^2),\qquad
E_3:\ Y^2=(V^2w^2+4U^2)(W^2w^2+4U^2),
\]
where $u=r+1/r$ and $w=r-1/r$. Peschmann eliminates his 1,072 fibers by forcing all rational points of $H_{m,n}$ to be the eight degenerate points through rank-zero and torsion information on these quotients. I use the same quotients when they are useful, but the 968 complementary fibers generally need stronger information. The result of \cite{Pes1072} is the starting point for the final count in Corollary~\ref{cor:2040}.

There is also a particularly useful direct necessary curve. To avoid conflicting with Peschmann's parameter $r=a/b$, I denote its parameter by $T$. If a master tuple is perfect, set
\[
T=\frac{UV_1}{U_1},\qquad y_U=\frac{UW_1}{U_1},\qquad
y_V=\frac{\sqrt M}{U_1},\qquad y_W=\frac{\sqrt F}{U_1}.
\]
Then all four quantities are positive and
\begin{equation}\label{eq:three-square}
y_U^2=T^2+U^2,\qquad y_V^2=T^2+V^2,\qquad y_W^2=T^2+W^2.
\end{equation}
Indeed the first identity is $W_1^2=U_1^2+V_1^2$, the second is the definition of $M$, and the third follows from $F/U_1^2=U^2(W_1/U_1)^2+V^2=T^2+U^2+V^2=T^2+W^2$. So every positive perfect specialisation lands on the smooth projective model of the affine curve
\begin{equation}\label{eq:genus-five}
C_{U,V,W}:\quad y_U^2=T^2+U^2,\quad y_V^2=T^2+V^2,\quad y_W^2=T^2+W^2.
\end{equation}
The map to the $T$-line has degree $8$. Over $\overline{\Q}$ its six branch values are $\pm iU,\pm iV,\pm iW$, and each contributes four simple ramification points, hence Riemann-Hurwitz gives $2g-2=8(-2)+24=8$ and $g=5$. There are eight rational points at $T=0$, obtained by independently choosing $y_U=\pm U$, $y_V=\pm V$, $y_W=\pm W$, and eight rational points over infinity, with $y_U/T,y_V/T,y_W/T\in\{\pm1\}$. I will call these the \emph{sixteen boundary points}. None comes from a positive master tuple.

For later use, \eqref{eq:genus-five} has four very simple elliptic quotients. For $\{d_i,d_j\}\subset\{U,V,W\}$ put
\[
E_{ij}:Y^2=X(X+d_i^2)(X+d_j^2),\qquad
P_{ij}=(T^2,T y_i y_j),
\]
and also
\[
E_A:Y^2=(X+U^2)(X+V^2)(X+W^2),\qquad
P_A=(T^2,y_Uy_Vy_W).
\]
On the positive chart these points have explicit rational halves:
\[
Q_{ij}=\bigl((T+y_i)(T+y_j),(T+y_i)(T+y_j)(y_i+y_j)\bigr),
\]
and
\[
Q_A=\bigl(T^2+y_Uy_V+y_Uy_W+y_Vy_W,\ (y_U+y_V)(y_U+y_W)(y_V+y_W)\bigr),
\]
with $2Q_{ij}=P_{ij}$ and $2Q_A=P_A$. These identities follow by direct substitution in the split cubics and the ordinary duplication formula. They are particularly useful later: a hypothetical cuboid does not merely give arbitrary elliptic points, but points whose halves have correlated $2$-Kummer classes.

\begin{lemma}[Kummer parity lemma]\label{lem:kummer}
Let $E:Y^2=(X-e_1)(X-e_2)(X-e_3)$ with distinct $e_i\in\Q$. Put $\delta_i(P)=X(P)-e_i$ in $\Q^\times/\Q^{\times2}$, with the usual root replacement $\delta_i((e_i,0))=(e_i-e_j)(e_i-e_k)$ and $\delta_i(O)=1$. Then $\delta=(\delta_1,\delta_2,\delta_3)$ is a homomorphism. Also, if $P=2Q$, every nonexceptional $X(P)-e_i$ is a rational square, with the root cases interpreted by the same replacement.
\end{lemma}

\begin{proof}
If a chord or tangent $\ell$ meets $E$ in $P,Q,R$ with $R=-(P+Q)$, then
\[
f(X)-\ell(X)^2=(X-X(P))(X-X(Q))(X-X(R)).
\]
Evaluating at $X=e_i$ shows $(X(P)-e_i)(X(Q)-e_i)(X(P+Q)-e_i)=\ell(e_i)^2$, which is exactly the homomorphism statement. The limiting root value gives the declared replacement. For doubling, direct simplification gives
\[
X(2Q)-e_i=
\frac{\bigl((X(Q)-e_i)^2-(e_i-e_j)(e_i-e_k)\bigr)^2}{4Y(Q)^2},
\]
and the torsion exceptions follow by continuity or direct substitution.
\end{proof}

A useful consequence, used repeatedly below, is the following. Suppose $G_1,\ldots,G_r$ and the full torsion generate $H\subset E(\Q)$, an algebraic rank upper bound $r$ is proved, and every nonzero parity combination of the $G_i$ plus every torsion translate has a nonsquare Kummer coordinate. Then the $G_i$ have rank $r$ and $[E(\Q):H]$ is finite and odd. Therefore $H$ has the same image as $E(\Q)$ in \emph{every} finite $2$-primary quotient. This is enough for the component/Kummer sieves below. Full odd-prime saturation is not needed.

\section{The first 848 fibers}

The first large part of the elimination is comparatively uniform. It is useful to replace \eqref{eq:genus-five} by the equivalent four-square coordinate $x=T^2+W^2$:
\begin{equation}\label{eq:four-square}
x,\quad x-U^2,\quad x-V^2,\quad x-W^2\quad\text{are rational squares}.
\end{equation}
An exact full-five $2$-descent on every one of the 968 fibers gives a finite Selmer survivor space equal to the rank-three subgroup generated by the known degenerate classes. Modulo the sign automorphisms there are only two relevant covering classes: the identity class and an infinity class. This is a \emph{necessary-class} statement, not by itself a point exhaustion. The remaining argument must therefore eliminate both classes.

The infinity class can be eliminated uniformly. Put $u=U/W$, $v=V/W$ and $c=u^2-v^2$. The corresponding mixed cover over $K=\Q(i)$ can be written
\[
r^2=\lambda^4+2c\lambda^2+1,\qquad
w^2=(\lambda^2+(u+iv)^2)(\lambda^2-(u-iv)^2).
\]
After $Z=\lambda/(u-iv)$ and $z=w/r$ one gets $z^2=(Z^2-1)/(Z^2+1)$, and hence the map
\[
(Z,z)\longmapsto (x,y)=\bigl(Z^2,Zz(Z^2+1)\bigr)
\]
to the fixed CM curve $E_0:y^2=x^3-x$. Fermat's classical descent gives $\rk E_0(\Q)=0$. The $(-1)$-twist is the same $j=1728$ curve, so the quadratic-field rank formula gives $\rk E_0(\Q(i))=0$. Reduction at split good primes bounds the torsion by $8$, and the eight visible points attain this bound. Their finite $x$-coordinates are $0,\pm1,\pm i$. The values $\pm i$ are not squares in $\Q(i)$, while $Z=\pm1,\pm i$ makes $\lambda=Z(u-iv)$ nonrational because $uv\ne0$. Thus rational $\lambda$ leaves only $\lambda=0,\infty$, both degenerate. So the infinity class has no nondegenerate rational point on any admissible fiber.

For the identity class there are two very small cross-covers. A projective form of \eqref{eq:genus-five} is
\[
A^2=B^2+4Z^2,\qquad P^2=B^2+4u^2Z^2,\qquad Q^2=B^2+4v^2Z^2,
\]
where on the affine chart $Z=1$ we have $B=2T/W$, $A=2y_W/W$, $P=2y_U/W$ and $Q=2y_V/W$. With $D=u^2-v^2$, the last two equations give the base conic
\[
u^2Q^2-v^2P^2=DB^2.
\]
Putting $s=(uQ-vP)/B$, a convenient projective parametrisation is
\[
B=2uvs,\qquad P=u(D-s^2),\qquad Q=v(D+s^2),
\]
and the two remaining square equations factor as
\[
A^2=(s^2-1)(s^2-D^2),\qquad
(2Z)^2=(s^2-(u+v)^2)(s^2-(u-v)^2).
\]
Hence an identity-class point lifts to both
\begin{equation}\label{eq:cross-cover}
C_\pm:\quad Y^2=(s^2-1)\bigl(s^2-(u\pm v)^2\bigr).
\end{equation}
For $b=(U\pm V)/W$, the Jacobian of $Y^2=(s^2-1)(s^2-b^2)$ is the split cubic
\begin{equation}\label{eq:cross-jacobian}
E_\pm:\quad y^2=(x-2W(U\pm V))(x+2W(U\pm V))
 \bigl(x-(W^2+(U\pm V)^2)\bigr).
\end{equation}
Whenever this elliptic curve has rank zero and torsion order $8$, the eight visible points on \eqref{eq:cross-cover} are all of them: two over $s=0$, two at infinity, and the four branch points $s=\pm1,\pm(U\pm V)/W$. They give $B=0$, $A=0$ or $Z=0$, so none is a cuboid.

There is also a useful whole-fiber argument. Write the variable Pythagorean pair as $R=a^2-b^2$, $S=2ab$, $H=a^2+b^2$. A perfect cuboid gives the same nonzero $(x,y)=(US,R)$ in the three concordant systems
\[
x^2+U^2y^2=\square,\qquad x^2+V^2y^2=\square,\qquad x^2+W^2y^2=\square.
\]
The usual split elliptic curve for a pair $(a,b)$ is $E_{ab}:Y^2=X(X+a^2)(X+b^2)$. On the rows used here exact $2$-descent gives rank zero and exact torsion $\Z/2\Z\times\Z/4\Z$, whose concordant-form classes are only the trivial ones, so the required nonzero simultaneous solution is impossible.

The counts are as follows: 320 fibers are eliminated by the whole-fiber concordant route, 396 by one of the rank-zero $C_\pm$ curves together with the uniform infinity argument, and 131 lie in both sets. Thus
\[
320+396-131=585
\]
fibers are eliminated by these two arguments.

A further 263 fibers are eliminated by strengthening the same small collection of ideas. It would be artificial to present these as 263 separate arguments, so I record the disjoint primary-method count instead:
\begin{equation}\label{eq:first-count}
\begin{array}{lr}
\toprule
\text{route} & \text{new fibers}\\
\midrule
\Q(i)\text{-trace rank-zero argument} & 96\\
\text{exact concordant/Ono-type rank-zero argument} & 93\\
C_\pm\text{ exact rank-zero argument} & 34\\
\text{negative-branch genus-two argument} & 12\\
\text{bielliptic quadratic-Chabauty argument} & 15\\
E_{uV}\text{ exact rank-zero argument} & 1\\
\text{isogeny/Gaussian Mordell-Weil and exact CRT arguments} & 12\\
\midrule
\text{total beyond the 585 core} & 263\\
\bottomrule
\end{array}
\end{equation}
so altogether these arguments eliminate $585+263=848$ fibers.

For completeness, here is what the less obvious labels in \eqref{eq:first-count} mean. The trace curve uses $k=W^2/(UV)$ and
\[
E_{\rm tr}:y^2=(x+8)(x-(8-8k))(x-(8+8k)).
\]
The identity-cover function satisfies a conjugation/translation relation over $\Q(i)$. If its value is rational, translating by rational $2$-torsion reduces the nonreal possibility to a nonreal root $\zeta=a_0+ib_0$. Separating real and imaginary parts gives $(a_0-1)^2=2k=W^2/[mn(m^2-n^2)]$. The denominator is the area of a primitive Pythagorean triangle and Fermat's square-area theorem excludes it. The remaining trace values are the degenerate ones. The negative branch starts from
\[
E^-_{AB}:y^2=x(x-A^2)(x-B^2)
\]
and keeps the omitted square as a lift condition. When a branch has genus two it is written
\[
C_{A,B,C}:y^2=(z^2+A)(z^2+B)(z^2+C),
\]
with the two degree-two elliptic quotient maps $(z,y)\mapsto(z^2,y)$ and $(z,y)\mapsto(s_3/z^2,s_3y/z^3)$, $s_3=ABC$, onto
\[
Y^2=(X+A)(X+B)(X+C),\qquad Y^2=(X+AB)(X+AC)(X+BC),
\]
respectively. Rank-one bases, all torsion cosets, exact finite local target sets and an ordinary $p$-adic height equation then give the terminal bielliptic calculation. For the twelve Gaussian/CRT rows, factoring the relevant sum of two squares over $\Q(i)$ reduces the problem to finitely many squareclasses and explicit quartic covers with elliptic Jacobians. Exact Mordell-Weil images modulo several good primes and CRT compatibility eliminate every non-anchor class, while the two anchor classes reduce to rank-zero quartic/cubic quotients containing only the known degeneracies. Bounded rational-point searches play no role in the completeness argument.

This leaves 120 fibers. They are more delicate, mainly because the short rank-zero quotient arguments above are no longer enough.

\section{The final 120 fibers}
The common tool for most of the final fibers is a correlated $p$-adic height sieve on the direct genus-five curve \eqref{eq:genus-five}. I state the exact form once, since the saturation and local-disk conditions are slightly easy to mishandle.

Let $E/\Q$ be one of the split quotients above, let $H\subset E(\Q)$ contain the complete torsion and have finite odd index, and fix a prime $q$. If the rational component group $\Phi_q$ has order $c=2^ah$ with $h$ odd, then
\setcounter{equation}{7}
\begin{equation}\label{eq:component-map}
E(\Q)\longrightarrow E(\Q)/2E(\Q)\times\Phi_q[2^\infty],
\qquad P\longmapsto\bigl(\delta(P),\operatorname{comp}(2hP)\bigr)
\end{equation}
has a finite $2$-primary image. By Lemma~\ref{lem:kummer}, the image from $H$ is the whole image: its quotient is at once of odd order, dividing $[E(\Q):H]$, and of $2$-power order. It is important, though, not to overread this statement. Only the $2$-primary component is projected here, while all compatible odd-component corrections are retained separately. It also does not assert that $H$ is saturated at every odd prime.

For a $q$-minimal model write $\ell_q(P)=d_q(P)+\gamma_q(P)$ with $d_q(P)=\max(0,-v_q(X(P)))$ and $\gamma_q$ the bounded component correction. I use the normalisation in which $\gamma_q=0$ on the identity component, with the opposite sign to the component correction $c_v$ in \cite{NV}. If $h_q$ kills the rational component group, set $\Psi_{h_q}=\psi_{h_q}^2$ and $\Phi_{h_q}=X\Psi_{h_q}-\psi_{h_q-1}\psi_{h_q+1}$. The division-polynomial valuation formula of Naskr\k{e}cki-Verzobio \cite{NV} gives
\[
\ell_q(P)=-\frac{\min(v_q(\Psi_{h_q}(P)),v_q(\Phi_{h_q}(P)))}{h_q^2}
\]
after the fixed normalisation, with the identity-component denominator term included. Thus exact valuation trees on the original $T$-line produce a finite conservative set $\Omega$ of all possible bad-prime height differences. Every integral residue class, every branch disk and every negative-valuation annulus is retained until either an exact nonsquare test removes it or leading-term dominance proves the infinite tail. An undecided test is retained, not discarded.

Now suppose two quotient curves $E_B,E_C$ have algebraic rank one. For a non-torsion $G$ put $\alpha=h_p(G)/\log_p(G)^2$. Rank one and quadraticity give $h_p(P)=\alpha\log_p(P)^2$ for every rational point, even if $G$ only generates an odd-index subgroup. At a good ordinary prime $p$, a point of $C_{U,V,W}$ must therefore satisfy on the \emph{same} $p$-adic disk
\begin{equation}\label{eq:height}
\rho(T)=\lambda_B(P_B)-\lambda_C(P_C)-\alpha_B L_B(P_B)^2+\alpha_C L_C(P_C)^2
 =-\sum_{q\ne p}\bigl(v_C(q)-v_B(q)\bigr)\log_p(q),
\end{equation}
where the right side belongs to the exact finite target set coming from $\Omega$. This is the basic height equation used below. It is the same rank-one $p$-adic height mechanism appearing in quadratic Chabauty \cite{Bianchi}, but here it is used only as a necessary equality on explicit elliptic quotients.

The terminal power-series tests are exact congruence tests, not numerical root searches. On every accepted chart the coefficient of $T^N$ in the omitted part of $\rho$ has
\[
v_p(\text{coefficient})\ge -4-2\lfloor\log_p(N+1)\rfloor.
\]
The all-order derivation of this bound is included in \url{proofs/PROOF_UNIFORM_HEIGHT_UPGRADES.md} in the companion package. With $p\ge11$, series degree $24$ and retained degrees $0,\ldots,22$, after the usual $T=pz$ rescaling and multiplication by $p^3$ every omitted term has valuation at least
\[
N-1-2\lfloor\log_p(N+1)\rfloor\ge20\qquad(N\ge23).
\]
Thus the finite polynomial gives the complete congruence modulo $p^{20}$. At a known $T=0$ or infinity boundary centre the exact constant and linear terms vanish and one factors \emph{$z^2$}, not $p^2z^2$. Every other point in that disk is still tested. A zero polynomial, depth limit or precision failure is retained as unresolved rather than discarded. A state is removed only when all of its actual disks are root-free after intersecting the height condition with the Kummer pattern. In particular, the Kummer and height conditions are always imposed on the same local disk.

\begin{proposition}[height/Kummer elimination principle]\label{prop:height}
Fix a fiber and suppose that: (i) the algebraic rank upper bounds and explicit points pass the parity/torsion tests of Lemma~\ref{lem:kummer}, (ii) the torsion subgroup is known exactly, (iii) the component images in \eqref{eq:component-map} are computed exactly, (iv) the bad-prime target sets $\Omega$ are complete, and (v) for each surviving global state, the exhaustive good-prime disk calculation for \eqref{eq:height} has empty terminal state set. Then $C_{U,V,W}(\Q)$ contains no finite nonzero point compatible with that state. If every nonboundary state is removed, the original $(m,n)$ fiber has no perfect cuboid.
\end{proposition}

\begin{proof}
Any rational point maps to rational points on all the quotient curves. The odd-index statement makes the recorded Kummer and $2$-primary component images exhaustive, while the division-polynomial formula makes $\Omega$ a conservative superset of every possible finite-place height contribution. The global $p$-adic height decomposition forces \eqref{eq:height} on the actual residue disk of the point. The all-order tail estimate makes each terminal polynomial test valid modulo $p^{20}$ on the whole disk, including branch and infinity charts. Therefore a rational point produces a global state which survives every one of the recorded intersections. If the final state set is empty there is no such point.
\end{proof}

Of the final 120 fibers, 46 are handled first. Forty-four use variants of Proposition~\ref{prop:height}. Depending on the fiber, I use three rank-one quotients, a direct calculation on the genus-five curve, the genus-three/CM bridge, or an extra square or component/Kummer condition. The remaining two, $(80,33)$ and $(94,93)$, use a purely algebraic higher $2$-isogeny rank-zero target. These 46 fibers are
\begin{small}
\begin{equation}\label{eq:forty-six}
\begin{gathered}
(28,27),(40,1),(43,8),(43,42),(48,25),(48,29),(49,48),(53,30),(61,14),(62,61),\\
(63,22),(64,7),(64,39),(65,12),(65,34),(67,12),(67,20),(67,66),(70,19),(72,37),\\
(73,14),(76,29),(79,12),(79,58),(80,33),(81,10),(83,36),(83,54),(83,58),(83,70),\\
(84,47),(85,66),(86,19),(87,28),(88,43),(91,72),(92,21),(92,47),(94,93),(95,12),\\
(95,72),(96,89),(97,20),(97,24),(97,68),(98,81).
\end{gathered}
\end{equation}
\end{small}
For the two algebraic rows one uses the positive-three-square map. A point of \eqref{eq:genus-five} gives $A^2=T^2+U^2$, $B^2=T^2+V^2$, $C^2=T^2+W^2$ and hence
\[
R=\frac{U^2V^2W^2}{T^2}+U^2V^2,\qquad
Y=\frac{U^2V^2W^2ABC}{T^3},
\]
on $Y^2=R(R+U^4)(R+V^4)$. Fisher's higher $2$-isogeny descent \cite{Fisher} proves rank zero for the exact curves attached to $(80,33)$ and $(94,93)$. Their complete torsion lists are then checked by exact group arithmetic and good-prime reduction. The only torsion points in the positive $R$ range force respectively
\[
T^2=11615439=3\cdot11\cdot47\cdot7489,\qquad
T^2=1626105=3\cdot5\cdot13\cdot31\cdot269,
\]
both impossible in $\Q$ because $v_3(T^2)$ would be odd. This eliminates all 46 fibers in \eqref{eq:forty-six}.

The other 74 fibers are a little cleaner: the geometric reductions are already in place, and it remains to prove the required rank or local obstruction exactly. They split as follows:
\begin{equation}\label{eq:seventy-four}
\begin{array}{lr}
\toprule
\text{method} & \text{fibers}\\
\midrule
E_{uV}\text{ higher }2\text{-isogeny descent} &35\\
E_{UW}\text{ central }L\text{-value enclosure} &23\\
C_\pm\text{ higher descent + full-five/CM bridge} &6\\
\text{rank-one height/Kummer sieve} &8\\
(71,64)\text{ exact }2\text{-primary projection at }2 &1\\
(89,40)\text{ full algebraic }8\text{-descent} &1\\
\midrule
\text{total} &74\\
\bottomrule
\end{array}
\end{equation}
I record the actual obstruction in each case, rather than hiding the argument behind a rank computation.

For 35 fibers, the $E_{uV}$ quotient is enough. A positive master tuple has $r=a/b>1$, $u=r+1/r>2$, and lies on
\[
Y^2=(V^2u^2+4(U^2-V^2))(W^2u^2-4V^2)=A^2u^4+Bu^2+C
\]
with $A=VW$, $B=4(W^2(U^2-V^2)-V^4)$ and $C=-16V^2(U^2-V^2)$. Put $X=Y+Au^2$, $x=2AX$ and $y=2A(X^2-C)/u$. Then
\begin{equation}\label{eq:euv-map}
y^2=(x+B)(x^2-4A^2C),\qquad u=\frac{x^2-4A^2C}{2Ay}.
\end{equation}
A rational $2$-isogeny takes this to the full-$2$-torsion model used for the exact Magma higher descent. Higher $2$-isogeny descent gives rank $0$ in all 35 cases, and the torsion subgroup has eight points. Their finite inverse $u$-values are only $\pm2$ and $\pm2V/W$, all $\le2$ in absolute value except for sign, while the other points are projective. This contradicts $u>2$.

For another 23 fibers there is an even more elementary point obstruction on $E_{UW}$ once rank zero is known. A perfect point gives
\[
P=(T^2,TAB)\in E_{UW}:\quad Y^2=X(X+U^2)(X+W^2),
\]
where $A^2=T^2+U^2$ and $B^2=T^2+W^2$. Set $S=T+A+B$ and $R=T A+T B+AB$. Then
\begin{equation}\label{eq:double-point}
Q=(T^2+R,TAB-SR)\in E_{UW}(\Q),\qquad 2Q=-P.
\end{equation}
The torsion group is exactly $\Z/2\Z\times\Z/4\Z$, whose doubled subgroup is $\{O,(0,0)\}$, so rank zero contradicts $P\in2E(\Q)$ with $X(P)=T^2>0$. Rank zero is proved here without GRH by a central-value enclosure. If $N$ is the conductor and the root number is $+1$, modularity gives
\[
L(E,1)=2\sum_{j\ge1}\frac{a_j}{j}e^{-2\pi j/\sqrt N}.
\]
The $a_j$ are exact integers, $|a_j|\le d(j)\sqrt j\le2j$, and with $z=e^{-2\pi/\sqrt N}$ the omitted tail after $k$ terms has absolute value at most $4z^{k+1}/(1-z)$. Arb ball arithmetic plus this explicit rational tail bound gives a strictly positive interval for all 23 curves. Modularity \cite{BCDT} and Kolyvagin's rank-zero theorem \cite{Kolyvagin} then imply $\rk E(\Q)=0$. Hence $\rk E(\Q)=0$ on all 23 fibers.

For the six $C_\pm$ rows
\[
(57,56)_{+},\ (42,41)_{-},\ (73,24)_{+},\ (88,83)_{-},\ (73,50)_{+},\ (97,66)_{-},
\]
the full-five descent first puts a hypothetical point into the identity or infinity class. In the identity class the same-source parameter
\[
s_0=\frac{U\alpha-V\gamma}{WT},\qquad
\alpha^2=T^2+V^2,\quad\gamma^2=T^2+U^2,
\]
lies on the appropriate $C_\pm$ of \eqref{eq:cross-cover}. The exact cubic map is
\[
s=\frac{2RY}{X^2-(2WR)^2},\qquad
q=\frac{R(X-2W^2)(X-2R^2)}{W(X^2-(2WR)^2)},\qquad R=U\pm V,
\]
from
\[
Y^2=(X-2WR)(X+2WR)(X-(W^2+R^2)).
\]
Higher $2$-isogeny descent gives rank zero. The complete torsion maps only to the six boundary values already listed, and the nondegenerate identities exclude all of them. The infinity class is the uniform $\Q(i)$ CM argument already given. This eliminates all six fibers.

Eight further fibers use the rank-one height/Kummer sieve:
\[
(38,29),(41,18),(64,49),(64,55),(73,64),(74,3),(92,13),(100,59).
\]
Their state counts are small enough to record explicitly:
\begin{equation}\label{eq:height-states}
\begin{array}{cllr}
\toprule
(m,n)&\text{rank-one pair}&\text{initial}&\text{successive }p:\text{states}\\
\midrule
(38,29)&E_{UV},E_{UW}&4&11:1,\ 13:0\\
(41,18)&E_{UW},E_{VW}&8&11:1,\ 13:0\\
(64,49)&E_{UV},E_{VW}&144&11:0\\
(64,55)&E_{UV},E_A&24&13:6,\ 19:0\\
(73,64)&E_{UV},E_A&24&11:8,\ 17:4,\ 19:2,\ 37:2,\ 41:2,\ 43:0\\
(74,3)&E_{UV},E_{VW}&4&13:0\\
(92,13)&E_{UW},E_{VW}&64&11:24,\ 17:4,\ 19:0\\
(100,59)&E_{UW},E_{VW}&32&11:10,\ 19:0\\
\bottomrule
\end{array}
\end{equation}
Each count is a count of conservative global Kummer/height states, not of rational points. Proposition~\ref{prop:height} says an actual rational point would give one. The final zero therefore excludes every positive specialisation. For $(73,64)$ the pair used is $(E_{UV},E_A)$.

The fiber $(71,64)$ needs one extra $2$-adic component refinement. Here $U=945$, $V=9088$ and $W=9137$. The rank-one pair is $(E_{UV},E_A)$. Ordinary primes reduce 768 joint states to six, but the prime $2$ has component orders with an odd factor: in particular the $E_{UV}$ component group has order $20$. Discarding the factor $5$ would be invalid. Instead \eqref{eq:component-map} projects by multiplying $2Q$ by the odd part and retains \emph{all} compatible odd-component corrections. If $T$ is $2$-adically integral and odd, then $T^2+945^2\equiv2\pmod8$, which is not a square in $\Q_2$. Thus the exact $2$-adic square conditions show either $v_2(T)<0$, giving $d_A-d_{UV}=-2$, or $T$ is even and both denominator terms are zero. For four of the final states the required value is $-2$, while the complete retained correction set is
\[
\{-1/5,9/5,11/5,3,21/5,5\}.
\]
For the other two the required value is $21/5$ while the retained set is $\{-2,0,6/5,14/5,16/5,24/5\}$. Neither target occurs, so all six states are eliminated. Notice that no $5$-saturation statement is needed.

The final fiber, $(89,40)$, has $(U,V,W)=(6321,7120,9521)$. The quotient
\[
E:Y^2=X(X+39955041)(X+90649441)
\]
receives a positive perfect point as $(X,Y)=(T^2,T y_Uy_W)$. A full algebraic $8$-descent in the sense of Fisher \cite{Fisher} gives rank zero. The complete torsion group has eight points. Its only positive $X$-coordinate is $60182241$, and
\[
7757^2=60171049<60182241<60186564=7758^2.
\]
An integral rational square is an integer square, so no torsion point has positive square $X$. Hence $X=T^2>0$ is impossible.

This completes the proof of Theorem~\ref{thm:main}. The finite computations use SageMath 10.9, PARI 2.17.3 and Magma 2.29-10. The companion package contains the exact scripts, inputs, valuation trees, component/Kummer images, local series, descent outputs and the final fiber ledger, so every computational step used above is reproducible from the displayed data.

\section{Final count}
The count is simply
\[
848+46+74=968.
\]
Thus every pair in $\mathcal R_{968}$ is eliminated. Together with Peschmann's 1,072 fibers, this gives all $2,040$ pairs in $\mathcal P_{100}$ and proves Corollary~\ref{cor:2040}.

Two technical points are worth keeping explicit. The resultant used above is $\operatorname{Res}(P,Q)=16U^8$; the different value printed in \cite{Pes1072} is a typo and does not affect Peschmann's genus-three argument. Also, none of the point searches used while finding the arguments is used as a completeness theorem here: the rank statements come from descent or $L$-value enclosures, while the $p$-adic exclusions use the complete component/Kummer images and the all-order tail bound described above.

So the final picture is quite clean. Peschmann eliminates 1,072 explicit fibers, and this paper eliminates the remaining 968 in the same $m\le100$ range. Each fixed fiber still contains infinitely many possible $(a,b)$, so this is considerably stronger than a finite edge search. It is still a bounded-fiber result, though. Extending it to arbitrary $(m,n)$ seems to require genuinely new uniform arithmetic, and I do not claim such a step here.

\section{Acknowledgements}

This paper, excluding the tables and a few short passages, is written by me with LaTeX assistance and polishing by GPT-5.6 Sol. The core 585-fiber argument is mine, while the remaining arguments are from GPT-5.6 Sol and GPT-6 Astra. The computational package was created by GPT-6 Astra; I deliberately instructed it to make the package minimal and well organised.

\end{document}